\documentclass[reqno]{amsart}

\usepackage[colorlinks,citecolor=green,linkcolor=red]{hyperref}
\usepackage{a4wide}
\usepackage{color}
\usepackage{mathrsfs}
\usepackage{mathtools}
\usepackage{amsmath}
\usepackage{cleveref}
\usepackage{amsthm}
\usepackage{amssymb}
\usepackage{esint}
\usepackage{nicefrac}
\numberwithin{equation}{section}

\usepackage[latin1]{inputenc}

\newcommand{\N}{\mathbb{N}}
\newcommand{\R}{\mathbb{R}}
\newcommand{\sfd}{{\sf d}}
\renewcommand{\d}{{\mathrm d}}

\newcommand{\restr}[1]{\lower3pt\hbox{\(|_{#1}\)}}
\newcommand{\eps}{\varepsilon}
\newcommand{\nchi}{{\raise.3ex\hbox{\(\chi\)}}}
\newcommand{\fr}{\penalty-20\null\hfill\(\blacksquare\)}
\newcommand{\mm}{{\mathfrak m}}

\newcommand{\LIP}{{\rm LIP}}

\newcommand{\lip}{{\rm lip}}

\DeclareMathOperator{\CD}{CD}
\DeclareMathOperator{\RCD}{RCD}

\newtheorem{theorem}{Theorem}[section]

\newtheorem{lemma}[theorem]{Lemma}
\newtheorem{proposition}[theorem]{Proposition}
\newtheorem{definition}[theorem]{Definition}

\newtheorem{remark}[theorem]{Remark}

\title[Rectifiability of \(\RCD(K,N)\) spaces]{Rectifiability of \(\RCD(K,N)\) spaces via $\delta$-splitting maps}
\author{Elia Bru\`{e}, Enrico Pasqualetto, and Daniele Semola}

\address{Scuola Normale Superiore\\
         Piazza dei Cavalieri 7 \\
         56126 Pisa \\
         Italy}
\email{elia.brue@sns.it}

\address{University of Jyvaskyla\\
         Department of Mathematics and Statistics \\
         P.O. Box 35 (MaD) \\
         FI-40014 University of Jyvaskyla \\
         Finland}
\email{enrico.e.pasqualetto@jyu.fi}

\address{Scuola Normale Superiore\\
         Piazza dei Cavalieri 7 \\
         56126 Pisa \\
         Italy}
\email{daniele.semola@sns.it}

\begin{document}
\date{\today} 
\date{\today} 

\begin{abstract}
In this note we give new proofs of rectifiability of $\RCD(K,N)$ spaces as metric measure spaces and lower semicontinuity of the essential dimension, via $\delta$-splitting maps. The arguments are inspired by the Cheeger-Colding theory for Ricci limits and rely on the second order differential calculus developed by Gigli and on the convergence and stability results by Ambrosio-Honda.   
\end{abstract}
\maketitle
%
%
%
%
\section*{Introduction}
In the last years the study of $\RCD(K,N)$ metric measure spaces has undergone a fast development. After the introduction of the \textit{curvature-dimension} condition $\CD(K,N)$ in the independent works \cite{Sturm06I,Sturm06II} and \cite{Lott-Villani09}, the notion of $\RCD(K,N)$ space was proposed in \cite{Gigli12} as a finite-dimensional counterpart of $\RCD(K,\infty)$, introduced in \cite{AmbrosioGigliSavare11-2} (see also \cite{AmbrosioGigliMondinoRajala12} for the case of $\sigma$-finite reference measure and \cite{BacherSturm10} for the introduction of the \textit{reduced curvature-dimension} condition $\CD^*(K,N)$). In the infinite-dimensional case the equivalence of the original Lagrangian approach with an Eulerian one, based on the Bochner inequality, was studied in \cite{AmbrosioGigliSavare12}. Then \cite{Erbar-Kuwada-Sturm13} established equivalence with the dimensional Bochner inequality for the so-called class $\RCD^*(K,N)$ (see also \cite{AmbrosioMondinoSavare13}). Equivalence between $\RCD^*(K,N)$ and $\RCD(K,N)$ has been eventually achieved in \cite{CavallettiMilman16} in the case of finite reference measure, closing the circle. 
Apart from smooth weighted Riemannian manifolds (with generalized Ricci tensor bounded from below), the $\RCD(K,N)$ class includes Ricci limit spaces, whose study was initiated by Cheeger-Colding in the nineties \cite{Cheeger-Colding97I,Cheeger-Colding97II,Cheeger-Colding97III} (see also the survey \cite{Cheeger01}), and Alexandrov spaces \cite{Petrunin11}. We refer the reader to the survey \cite{AmbrosioICM} for an account about this quickly developing research area. 

Many efforts have been recently aimed at understanding the structure theory of $\RCD(K,N)$ spaces. After \cite{Mondino-Naber14} by Mondino-Naber, we know that they are rectifiable as metric spaces and later, in the three independent works by De Philippis-Marchese-Rindler, Kell-Mondino and Gigli together with the second named author \cite{DPMR16,MK16,GP16-2}, the analysis was sharpened taking into account the behaviour of the reference measure and getting rectifiability as metric measure spaces. Moreover, in the recent \cite{BS18}, the first and the third named authors proved that $\RCD(K,N)$ spaces have constant dimension, in the almost everywhere sense.

The development of this theory was inspired in turn by the results obtained for Ricci limit spaces in the seminal papers by Cheeger-Colding (see also \cite{ColdingNaber12} by Colding-Naber for constancy of dimension).\\ 
In the proofs given in \cite{Cheeger-Colding97I,Cheeger-Colding97III} a crucial role was played by $(k,\delta)$-splitting maps: 
\begin{definition}
	Let \((X,\sfd,\mm)\) be an \(\RCD(-1,N)\) space. Let \(x\in X\) and \(\delta>0\)
	be given. Then a map \(u=(u_1,\ldots,u_k)\colon B_r(x)\to\R^k\) is said to be
	a \emph{\((k,\delta)\)-splitting map} provided:
	\begin{itemize}
		\item[\(\rm i)\)] \(u_a\colon B_r(x)\to\R\) is harmonic and \(C_N\)-Lipschitz
		for every \(a=1,\ldots,k\),
		\item[\(\rm ii)\)] \(r^2\fint_{B_r(x)}\big|{\rm Hess}(u_a)\big|^2\,\d\mm\leq\delta\)
		for every \(a=1,\ldots,k\),
		\item[\(\rm iii)\)] \(\fint_{B_r(x)}|\nabla_\mm u_a\cdot\nabla_\mm u_b
		-\delta_{ab}|\,\d\mm\leq\delta\) for every \(a,b=1,\ldots,k\).
	\end{itemize}
\end{definition}
These maps provide approximations, in the integral $L^2$-sense and up to the second order, of $k$ independent coordinate functions in the Euclidean space. They were introduced in \cite{Cheeger-Colding96}, in the study of Riemannian manifolds with lower Ricci curvature bounds.\\
Item ii) in the definition of $\delta$-splitting maps is about smallness of the $L^2$-norm of the Hessian, in scale invariant sense. In \cite{Cheeger-Colding97I,Cheeger-Colding97III} and in more recent works dealing with Ricci limits as \cite{CheegerJiangNaber18}, $\delta$-splitting maps are built only at the level of the smooth approximating sequence, where there is a clear notion of Hessian available, the metric information they encode ($\eps$-GH closeness to Euclidean spaces) is then passed to the limit.\\ 
Prior than \cite{Gigli14}, there was no notion of Hessian available in the $\RCD$ framework. This, together with the absence of smooth approximating sequences, motivated the necessity to find an alternative approach to rectifiability in \cite{Mondino-Naber14,DPMR16,MK16,GP16-2} with respect to the Cheeger-Colding theory. A new \textit{almost splitting via excess} theorem was the main ingredient playing the role of the theory of $\delta$-splitting maps in \cite{Mondino-Naber14} while, studying the behaviour of the reference measure with respect to charts, a crucial role was played in both \cite{DPMR16,MK16,GP16-2}, by a recent and powerful result obtained by De Philippis-Rindler \cite{DPR}.

Nowadays we have at our disposal both a second order differential calculus on $\RCD$ spaces \cite{Gigli14} and general convergence and stability results for Sobolev functions on converging sequences of $\RCD(K,N)$ spaces \cite{AmbrosioHonda17,AmbrosioHonda18}.
In our previous paper \cite{BPS19} we exploited all these tools to prove rectifiability for reduced boundaries of sets of finite perimeter in this context. The study of \cite{BPS19} was devoted to the theory in codimension one, which required some additional ideas and technical efforts, but it was evident that similar arguments could provide new and more direct proofs of rectifiability for $\RCD(K,N)$ spaces in the spirit of those in \cite{Cheeger-Colding97I,Cheeger-Colding97III}. 

Taking as a starting point existence of Euclidean tangents almost everywhere with respect to the reference measure, obtained by Gigli-Mondino-Rajala in \cite{Gigli-Mondino-Rajala15}, in this short note we provide the arguments to get uniqueness (almost everywhere) of tangents and rectifiability of $\RCD(K,N)$ spaces as metric measure spaces via $\delta$-splitting maps. Moreover, we recover via a different strategy the result about lower semicontinuity of the so called essential dimension proved firstly in \cite{Kitabeppu2019}. 
After Section \ref{sec:prel}, dedicated to review some preliminaries and establish the basic tool about propagation of the $\delta$-splitting property, the remaining Subsection \ref{subsec:uniq}, Subsection \ref{subsec:rect} and Subsection \ref{subsec:measure} are devoted to uniqueness of tangents and lower semicontinuity of the essential dimension, metric rectifiability and the behaviour of the reference measure under charts, respectively.

\subsection*{Acknowledgements} 
The authors wish to thank Luigi Ambrosio, Nicola Gigli, Andrea Mondino and Tapio Rajala for useful comments on an earlier version of this note.\\
The second named author was partially supported by the Academy of Finland, projects 307333 and 314789.
Part of this work was developed while the first and third named authors were visiting the Department of Mathematics and Statistics of the University of Jyvaskyla: they wish to thank the institute for the excellent working conditions and the stimulating atmosphere.

\section{Preliminaries and notation}\label{sec:prel}
\subsection{Differential calculus on metric measure spaces}
For our purposes, a \emph{metric measure space} is a triple \((X,\sfd,\mm)\),
where \((X,\sfd)\) is a proper metric space, while \(\mm\geq 0\)
is a Radon measure on \(X\). Given a Lipschitz function
\(f\colon X\to\R\), we will denote by \(\lip(f)\colon X\to[0,+\infty)\) its
\emph{slope}, which is the function defined as
\[
\lip(f)(x)\coloneqq\varlimsup_{y\to x}\frac{\big|f(x)-f(y)\big|}{\sfd(x,y)}
\quad\text{ for every accumulation point }x\in X
\]
and \(\lip(f)(x)\coloneqq 0\) elsewhere. Given any open set \(\Omega\subseteq X\),
we denote by \(\LIP_{\rm c}(\Omega)\) the family of all Lipschitz functions
\(f\colon\Omega\to\R\) whose support is bounded and satisfies
\({\rm dist}\big({\rm spt}(f),X\setminus\Omega\big)>0\).
\subsubsection{Sobolev space}
Following \cite{Cheeger00}, we define the \emph{Sobolev space} \(H^{1,2}(X,\sfd,\mm)\) as
\[
H^{1,2}(X,\sfd,\mm)\coloneqq\big\{f\in L^2(\mm)\;\big|\;{\rm Ch}(f)<+\infty\big\},
\]
where the \emph{Cheeger energy} \({\rm Ch}\colon L^2(\mm)\to[0,+\infty]\)
is the convex, lower semicontinuous functional
\[
{\rm Ch}(f)\coloneqq\inf\bigg\{\varliminf_{n\to\infty}\int\lip^2(f_n)\,\d\mm
\;\bigg|\;(f_n)_n\subseteq L^2(\mm)\text{ bounded Lipschitz,}\,
\lim_{n\to\infty}\|f_n-f\|_{L^2(\mm)}=0\bigg\}.
\]
It holds that \(H^{1,2}(X,\sfd,\mm)\) is a Banach space if endowed with the norm
\(\|\cdot\|_{H^{1,2}(X,\sfd,\mm)}\), given by
\[
\|f\|_{H^{1,2}(X,\sfd,\mm)}\coloneqq\Big(\|f\|_{L^2(\mm)}^2+{\rm Ch}(f)\Big)^{1/2}
\quad\text{ for every }f\in H^{1,2}(X,\sfd,\mm).
\]
Given any \(f\in H^{1,2}(X,\sfd,\mm)\), one can select a canonical object
\(|D_\mm f|\in L^2(\mm)\) -- called the \emph{minimal relaxed slope} of \(f\)
-- for which \({\rm Ch}(f)\) admits the integral representation
\({\rm Ch}(f)=\int|D_\mm f|^2\,\d\mm\).\\
We have chosen to stress the dependence on the measure for the gradient and the other differential objects, here and in the sequel, to avoid confusion. 

\medskip

Given an open set \(\Omega\subseteq X\), we define \(H^{1,2}_{\rm loc}(\Omega,\sfd,\mm)\)
as the space of all those \(f\in L^2_{\rm loc}(\mm)\) such that
\(\eta f\in H^{1,2}(X,\sfd,\mm)\)
holds for every \(\eta\in\LIP_{\rm c}(\Omega)\).
Thanks to the locality property of the minimal relaxed slope, it makes sense to define
\(|D_\mm f|\in L^2_{\rm loc}(\mm)\) as
\[
|D_\mm f|\coloneqq\big|D_\mm(\eta f)\big|\;\;\mm\text{-a.e.\ on }\{\eta=1\},
\quad\text{ for any }\eta\in\LIP_{\rm c}(\Omega).
\]
Finally, we define \(H^{1,2}(\Omega,\sfd,\mm)\) as the space of all
\(f\in H^{1,2}_{\rm loc}(\Omega,\sfd,\mm)\) such that \(f,|D_\mm f|\in L^2(\mm)\).
\subsubsection{Tangent module}
Whenever \(H^{1,2}(X,\sfd,\mm)\) is a Hilbert space, we will say that
\((X,\sfd,\mm)\) is \emph{infinitesimally Hilbertian}. In this case,
we recall from \cite{Gigli14} that the \emph{tangent module} \(L^2(TX)\) and the
corresponding \emph{gradient} map \(\nabla_\mm\colon H^{1,2}(X,\sfd,\mm)\to L^2(TX)\)
can be characterised as follows: \(L^2(TX)\) is an \(L^2(\mm)\)-normed
\(L^\infty(\mm)\)-module (in the sense of \cite[Definition 1.3]{Gigli17})
that is generated by \(\big\{\nabla_\mm f\,:\,f\in H^{1,2}(X,\sfd,\mm)\big\}\),
while \(\nabla_\mm\) is a linear map satisfying \(|\nabla_\mm f|=|D_\mm f|\)
\(\mm\)-a.e.\ on \(X\) for all \(f\in H^{1,2}(X,\sfd,\mm)\).
The pointwise scalar product
\(L^2(TX)\times L^2(TX)\ni(v,w)\mapsto v\cdot w\in L^1(\mm)\),
\[
v\cdot w\coloneqq\frac{|v+w|^2-|v|^2-|w|^2}{2}\quad\text{ for every }v,w\in L^2(TX),
\]
is a symmetric bilinear form, as a consequence of the infinitesimal Hilbertianity
assumption.

The dual module of \(L^2(TX)\) is denoted by \(L^2(T^*X)\) and
called the \emph{cotangent module} of \(X\).
\medskip

In the framework of weighted Euclidean spaces, we have another notion
of tangent module at our disposal. Given a Radon measure \(\nu\geq 0\)
on \(\R^k\), we denote by \(L^2(\R^k,\R^k;\nu)\) the space of all \(L^2(\nu)\)-maps
from \(\R^k\) to itself. It turns out that \(L^2(\R^k,\R^k;\nu)\) is an
\(L^2(\nu)\)-normed \(L^\infty(\nu)\)-module generated by
\(\big\{\nabla f\,:\,f\in C^\infty_c(\R^k)\big\}\), where
\(\nabla f\colon\R^k\to\R^k\) stands for the `classical' gradient of \(f\).
\subsubsection{Divergence and Laplacian}
In the setting of infinitesimally Hilbertian spaces \((X,\sfd,\mm)\),
one can consider the following notions of divergence and Laplacian:
\begin{itemize}
\item \textsc{Divergence.} We declare that \(v\in L^2(TX)\) belongs to
\(D({\rm div}_\mm)\) provided there exists a (uniquely determined)
function \({\rm div}_\mm(v)\in L^2(\mm)\) such that
\[
\int\nabla_\mm f\cdot v\,\d\mm=-\int f\,{\rm div}_\mm(v)\,\d\mm
\quad\text{ for every }f\in H^{1,2}(X,\sfd,\mm).
\]
\item \textsc{Laplacian.} Given any open set \(\Omega\subseteq X\),
we declare that \(f\in H^{1,2}(\Omega,\sfd,\mm)\) belongs to \(D(\Omega,\Delta_\mm)\)
provided there exists a (uniquely determined) function \(\Delta_\mm f\in L^2(\Omega)\)
such that
\[
\int_\Omega\nabla_\mm f\cdot\nabla_\mm g\,\d\mm=-\int_\Omega g\,\Delta_\mm f\,\d\mm
\quad\text{ for every }g\in H^{1,2}_0(\Omega,\sfd,\mm),
\]
where \(H^{1,2}_0(\Omega,\sfd,\mm)\) stands for the closure of \(\LIP_{\rm c}(\Omega)\)
in \(H^{1,2}(X,\sfd,\mm)\). For brevity, we shall write \(D(\Delta_\mm)\) in place
of \(D(X,\Delta_\mm)\).
\end{itemize}
The domains \(D({\rm div}_\mm)\) and \(D(\Omega,\Delta_\mm)\)
are vector subspaces of \(L^2(TX)\) and \(H^{1,2}(\Omega,\sfd,\mm)\), respectively.
Moreover, the operators \({\rm div}_\mm\colon D({\rm div}_\mm)\to L^2(\mm)\)
and \(\Delta_\mm\colon D(\Omega,\Delta_\mm)\to L^2(\Omega)\) are linear.

It can be readily checked that a given function \(f\in H^{1,2}(X,\sfd,\mm)\)
belongs to \(D(\Delta_\mm)\) if and only if its gradient \(\nabla_\mm f\)
belongs to \(D({\rm div}_\mm)\). In this case, it also holds that
\(\Delta_\mm f={\rm div}_\mm(\nabla_\mm f)\).
\subsection{\texorpdfstring{\(\rm RCD\)}{RCD} spaces}
We assume the reader to be familiar with the language of \(\RCD(K,N)\) spaces
and the notion of pointed measured Gromov--Hausdorff convergence (often
abbreviated to pmGH).

We recall the following scaling property: if \((X,\sfd,\mm)\) is an \(\RCD(K,N)\)
space, then \((X,\sfd/r,\lambda\mm)\) is an \(\RCD(r^2K,N)\) space for any choice
of \(r,\lambda>0\). Furthermore,
there exists a distance \(\sfd_{\rm pmGH}\) on the set (of isomorphism classes)
of \(\RCD(K,N)\) spaces that metrises the pmGH-topology \cite{Gigli-Mondino-Savare13}.
\begin{remark}\label{rmk:cptness_RCD}{\rm
		Any sequence \((X_n,\sfd_n,\mm_n,x_n)\), \(n\in\N\) of pointed
		\(\RCD(K,N)\) spaces converges, up to the extraction of a subsequence, to some pointed \(\RCD(K,N)\)
		space \((X,\sfd,\mm,x)\) with respect to the pmGH-topology.
		This follows from a compactness argument due to Gromov
		and the stability of the \(\RCD(K,N)\) condition.
		\fr}\end{remark}

\subsubsection{Test functions}
Let \((X,\sfd,\mm)\) be an \(\RCD(K,N)\) space. A fundamental class of Sobolev
functions on \(X\) is given by the algebra of \emph{test functions} \cite{Gigli14}:
\[
{\rm Test}(X)\coloneqq\Big\{f\in D(\Delta_\mm)\cap L^\infty(\mm)\;\Big|\;
|D_\mm f|\in L^\infty(\mm),\;\Delta_\mm f\in H^{1,2}(X,\sfd,\mm)\cap L^\infty(\mm)\Big\}.
\]
Since \(\RCD\) spaces enjoy the Sobolev-to-Lipschitz property, we know that
any element of \({\rm Test}(X)\) admits a Lipschitz representative. Moreover,
it holds that \({\rm Test}(X)\) is dense in \(H^{1,2}(X,\sfd,\mm)\) and
that \(\nabla_\mm f\cdot\nabla_\mm g\in H^{1,2}(X,\sfd,\mm)\) for every
\(f,g\in{\rm Test}(X)\).
\begin{lemma}[Good cut-off functions \cite{AmbrosioMondinoSavare13-2,Mondino-Naber14}]
\label{lem:good_cut-off}
Let \((X,\sfd,\mm)\) be an \(\RCD(K,N)\) space. Let $0<r<R$ and $x\in X$. Then there exists
\(\eta\in{\rm Test}(X)\) such that \(0\leq\eta\leq 1\) on \(X\), the support
of \(\eta\) is compactly contained in $B_R(x)$, and \(\eta=1\) on $B_r(x)$.
\end{lemma}

We recall the notion of \emph{Hessian} of a test function \cite{Gigli14}:
given \(f\in{\rm Test}(X)\), we denote by \({\rm Hess}(f)\) the unique element
of the tensor product \(L^2(T^*X)\otimes L^2(T^*X)\)
(cf.\ \cite[Section 1.5]{Gigli14}) such that
\[\begin{split}
&\,2\int h\,{\rm Hess}(\nabla_\mm g_1\otimes\nabla_\mm g_2)\,\d\mm\\
=&-\int\nabla_\mm f\cdot\nabla_\mm g_1\,{\rm div}_\mm(h\nabla_\mm g_2)
+\nabla_\mm f\cdot\nabla_\mm g_2\,{\rm div}_\mm(h\nabla_\mm g_1)
+h\nabla_\mm f\cdot\nabla_\mm(\nabla_\mm g_1\cdot\nabla_\mm g_2)\,\d\mm
\end{split}\]
holds for every \(h,g_1,g_2\in{\rm Test}(X)\). The pointwise norm
\(\big|{\rm Hess}(f)\big|\) of \({\rm Hess}(f)\) belongs to \(L^2(\mm)\).
\medskip

Given an open set \(\Omega\subseteq X\)
and a function \(f\in D(\Omega,\Delta_\mm)\), we say that \(f\) is \emph{harmonic}
if \(\Delta_\mm f=0\). If in addition \(f\) is Lipschitz, then one can define
(the modulus of) its Hessian as follows:
\begin{equation}\label{eq:def_Hess_local}
\big|{\rm Hess}(f)\big|\coloneqq\big|{\rm Hess}(\eta f)\big|\;\;\mm\text{-a.e.\ on }
\{\eta=1\},\quad\text{ for every }\eta\in{\rm Test}(X)
\text{ with }{\rm spt}(\eta)\subseteq\Omega.
\end{equation}
This way we obtain a well-defined function
\(\big|{\rm Hess}(f)\big|\colon\Omega\to[0,+\infty)\),
thanks to the locality property of the Hessian and the fact that
\(\eta f\in{\rm Test}(X)\) for every \(\eta\) as in \eqref{eq:def_Hess_local}.
\subsection{Splitting maps on \texorpdfstring{\(\rm RCD\)}{RCD} spaces}
\label{s:splitting_maps}
In this subsection we collect the main properties of $\delta$-splitting maps that we will use in the sequel. Let us recall that their introduction in the study of spaces with lower Ricci curvature bounds dates back to \cite{Cheeger-Colding96} and that they have been extensively used in \cite{Cheeger-Colding97I,Cheeger-Colding97II,Cheeger-Colding97III} and in more recent works on Ricci limits \cite{CheegerNaber15,CheegerJiangNaber18}. Before the development of the second order calculus in \cite{Gigli14}, there was need for alternative arguments avoiding the use of the Hessian in order to develop the structure theory of $\RCD(K,N)$ spaces in \cite{Mondino-Naber14}. In recent times (see \cite{AHPT18,BPS19}) $\delta$-splitting maps have come into play also in the $\RCD$-theory thanks to \cite{Gigli14} and the stability results of \cite{AmbrosioHonda17,AmbrosioHonda18}.

The results connecting $\delta$-splitting maps with $\eps$-isometries stated below are borrowed from \cite{BPS19}. Although being less local than those provided by the Cheeger-Colding theory, they are sufficient for our purposes and allow for more direct proofs via compactness.

\begin{definition}[Splitting map \cite{BPS19}]\label{def:splitting maps}
	Let \((X,\sfd,\mm)\) be an \(\RCD(-1,N)\) space. Let \(x\in X\) and \(\delta>0\)
	be given. Then a map \(u=(u_1,\ldots,u_k)\colon B_r(x)\to\R^k\) is said to be
	a \emph{\(\delta\)-splitting map} provided:
	\begin{itemize}
		\item[\(\rm i)\)] \(u_a\colon B_r(x)\to\R\) is harmonic and \(C_N\)-Lipschitz
		for every \(a=1,\ldots,k\),
		\item[\(\rm ii)\)] \(r^2\fint_{B_r(x)}\big|{\rm Hess}(u_a)\big|^2\,\d\mm\leq\delta\)
		for every \(a=1,\ldots,k\),
		\item[\(\rm iii)\)] \(\fint_{B_r(x)}|\nabla_\mm u_a\cdot\nabla_\mm u_b
		-\delta_{ab}|\,\d\mm\leq\delta\) for every \(a,b=1,\ldots,k\).
	\end{itemize}
\end{definition}
\begin{proposition}[From GH-isometry to \(\delta\)-splitting \cite{BPS19}]
	\label{prop:GH_to_delta-split}
	Let \(N>1\) be given. Then for any \(\delta>0\) there exists
	\(\eps=\eps_{N,\delta}>0\) such that the following property holds.
	If \((X,\sfd,\mm)\) is an \(\RCD(K,N)\) space, \(x\in X\),
	\(r>0\) with \(r^2|K|\leq\eps\), and there is an
	\(\RCD(0,N-k)\) space \((Z,\sfd_Z,\mm_Z,z)\) such that
	\[
	\sfd_{\rm pmGH}\Big(\big(X,\sfd/r,\mm^r_x,x\big),
	\big(\R^k\times Z,\sfd_{\rm Eucl}\times\sfd_Z,
	\mathcal L^k\otimes\mm_Z,(0^k,z)\big)\Big)\leq\eps,
	\]
	then there exists a \(\delta\)-splitting map \(u\colon B_{5r}(x)\to\R^k\).
\end{proposition}
\begin{proposition}[From \(\delta\)-splitting to GH-isometry \cite{BPS19}]
	\label{prop:delta-split_to_GH}
	Let \(N>1\) be given. Then for any \(\eps>0\) there exists \(\delta=\delta_{N,\eps}>0\)
	such that the following property holds. If \((X,\sfd,\mm)\) is an \(\RCD(K,N)\)
	space, \(x\in X\), and there exists a map \(u\colon B_r(x)\to\R^k\) such that
	\(u\colon B_s(x)\to\R^k\) is a \(\delta\)-splitting map for all \(s<r\), then
	for any \((Y,\varrho,\mathfrak n,y)\in{\rm Tan}_x(X,\sfd,\mm)\) it holds that
	\[
	\sfd_{\rm pmGH}\Big((Y,\varrho,\mathfrak n,y),\big(\R^k\times Z,\sfd_{\rm Eucl}\times\sfd_Z,
	\mathcal L^k\otimes\mm_Z,(0^k,z)\big)\Big)\leq\eps,
	\]
	for some pointed \(\RCD(0,N-k)\) space \((Z,\sfd_Z,\mm_Z,z)\).
\end{proposition}
Below we state and prove a result about propagation of the $\delta$-splitting property at many locations with respect to the reference measure and at all scales. The proof is based on a standard maximal function argument.
\begin{proposition}[Propagation of the \(\delta\)-splitting property]
	\label{prop:propag_delta-split}
	Let \(N>1\) be given. Then there exists a constant \(C_N>0\) such that
	the following property holds. If \((X,\sfd,\mm)\) is an \(\RCD(K,N)\)
	space and \(u\colon B_{2r}(p)\to\R^k\) is a
	\(\delta\)-splitting map for some \(p\in X\), \(r>0\) with \(r^2|K|\leq 1\), and
	\(\delta\in(0,1)\), then there exists a Borel set \(G\subseteq B_r(p)\) such that
	\(\mm\big(B_r(p)\setminus G\big)\leq C_N\sqrt\delta\,\mm\big(B_r(p)\big)\) and
	\[
	u\colon B_s(x)\to\R^k\text{ is a }\sqrt\delta\text{-splitting map,}
	\quad\text{ for every }x\in G\text{ and }s\in(0,r).
	\]
\end{proposition}
\begin{proof}
	Thanks to a scaling argument, it is sufficient to prove the claim for \(r=1\) and \(|K|\leq 1\). Let us define \(G\subseteq B_1(p)\)
	as \(G\coloneqq\bigcap_{a=1}^k G_a\cap\bigcap_{a,b=1}^k G_{a,b}\), where we set
	\[\begin{split}
	G_a&\coloneqq\bigg\{x\in B_1(p)\;\bigg|\;\sup_{s\in(0,1)}
	\fint_{B_s(x)}\big|{\rm Hess}(u_a)\big|^2\,\d\mm\leq\sqrt\delta\bigg\},\\
	G_{a,b}&\coloneqq\bigg\{x\in B_1(p)\;\bigg|\;\sup_{s\in(0,1)}
	\fint_{B_s(x)}\big|\nabla_\mm u_a\cdot\nabla_\mm u_b-\delta_{ab}\big|\,\d\mm
	\leq\sqrt\delta\bigg\}.
	\end{split}\]
	It holds that \(u\colon B_s(x)\to\R^k\) is a \(\sqrt\delta\)-splitting
	map for all \(x\in G\) and \(s\in(0,1)\). To prove the claim,
	it remains to show that \(\mm\big(B_1(p)\setminus G_a\big),
	\mm\big(B_1(p)\setminus G_{a,b}\big)\leq C_N\sqrt\delta\,\mm\big(B_1(p)\big)\)
	for all \(a,b=1,\ldots,k\).
	
	Given any \(x\in B_1(p)\setminus G_a\), we can choose \(s_x\in(0,1)\)
	such that \(\fint_{B_{s_x}(x)}\big|{\rm Hess}(u_a)\big|^2\,\d\mm>\sqrt\delta\).
	By Vitali covering lemma, we can find a sequence
	\((x_i)_i\subseteq B_1(p)\setminus G_a\) such that
	\(\big\{B_{s_{x_i}}(x_i)\big\}_i\) are pairwise disjoint and
	\(B_1(p)\setminus G_a\subseteq\bigcup_i B_{5s_{x_i}}(x_i)\).
	Therefore, it holds that
	\[\begin{split}
	\mm\big(B_1(p)\setminus G_a\big)&\leq\sum_{i\in\N}\mm\big(B_{5s_{x_i}}(x_i)\big)
	\leq C_N\sum_{i\in\N}\mm\big(B_{s_{x_i}}(x_i)\big)\leq
	\frac{C_N}{\sqrt\delta}\sum_{i\in\N}\int_{B_{s_{x_i}}(x_i)}
	\big|{\rm Hess}(u_a)\big|^2\,\d\mm\\
	&\leq\frac{C_N\mm\big(B_2(p)\big)}{\sqrt\delta}
	\fint_{B_2(p)}\big|{\rm Hess}(u_a)\big|^2\,\d\mm
	\leq C_N\sqrt\delta\,\mm\big(B_1(p)\big),
	\end{split}\]
	where we used the doubling property of \(\mm\), the defining property of \(s_{x_i}\)
	and the fact that \(u\) is a \(\delta\)-splitting map on \(B_2(p)\).
	An analogous argument shows that \(\mm\big(B_1(p)\setminus G_{a,b}\big)\leq
	C_N\sqrt\delta\,\mm\big(B_1(p)\big)\) for all \(a,b=1,\ldots,k\),
	thus the statement is achieved.
\end{proof}
\section{Structure theory for \texorpdfstring{\(\rm RCD\)}{RCD} spaces}
Given a pointed \(\RCD(K,N)\) space \((X,\sfd,\mm,x)\) and a radius \(r\in(0,1)\),
we define the normalised measure \(\mm_r^x\) on \(X\) as
\[
\mm_r^x\coloneqq\frac{\mm}{C(x,r)},\quad\text{ where }
C(x,r)\coloneqq\int_{B_r(x)}\left(1-\frac{\sfd(\cdot,x)}{r}\right)\,\d\mm.
\]
The \emph{tangent cone} \({\rm Tan}_x(X,\sfd,\mm)\)
is defined as the family of all those spaces \((Y,\varrho,\mathfrak n,y)\) such that
\[
\lim_{n\to\infty}\sfd_{\rm pmGH}\big((X,\sfd/r_n,\mm_{r_n}^x,x),
(Y,\varrho,\mathfrak n,y)\big)=0
\]
for some sequence \((r_n)_n\subseteq(0,1)\) of radii with \(r_n\searrow 0\).
It follows from the scaling property of the \(\RCD\) condition and Remark
\ref{rmk:cptness_RCD} that any element of \({\rm Tan}_x(X,\sfd,\mm)\) is
a pointed \(\RCD(0,N)\) space.
Let us briefly recall the properties that we take as a starting point for our analysis of the structure theory of $\RCD(K,N)$ spaces. The first one is a version of the \textit{iterated tangent property} suited for this setting. Building upon this, in \cite{Gigli-Mondino-Rajala15} it was proved that at almost every point there exists at least one Euclidean space in the tangent cone, on $\RCD(K,N)$ spaces (see Theorem \ref{thm:Eucl_tgs} below).
\begin{theorem}[Iterated tangent property \cite{Gigli-Mondino-Rajala15}]\label{thm:iterated_tg}
	Let \((X,\sfd,\mm)\) be an \(\RCD(K,N)\) space. Then for
	\(\mm\)-a.e.\ point \(x\in X\) it holds that
	\[
	{\rm Tan}_z(Y,\varrho,\mathfrak n)\subseteq{\rm Tan}_x(X,\sfd,\mm)
	\quad\text{ for every }(Y,\varrho,\mathfrak n,y)\in{\rm Tan}_x(X,\sfd,\mm)
	\text{ and }z\in Y.
	\]
\end{theorem}
\begin{theorem}[Euclidean tangents to \(\RCD\) spaces \cite{Gigli-Mondino-Rajala15}]
	\label{thm:Eucl_tgs}
	Let \((X,\sfd,\mm)\) be an \(\RCD(K,N)\) space. Then for \(\mm\)-a.e.\ point
	\(x\in X\) there exists \(k(x)\in\N\) with \(k(x)\leq N\) such that
	\[
	\big(\R^{k(x)},\sfd_{\rm Eucl},c_{k(x)}\mathcal L^{k(x)},0^{k(x)}\big)
	\in{\rm Tan}_x(X,\sfd,\mm),
	\]
	where we set \(c_k\coloneqq\mathcal L^k\big(B_1(0^k)\big)/(k+1)\) for every \(k\in\N\).
\end{theorem}

\subsection{Uniqueness of tangent cones}\label{subsec:uniq}
Let \((X,\sfd,\mm)\) be an \(\RCD(K,N)\) space. Then we define
\[
\mathcal R_k\coloneqq\Big\{x\in X\;\Big|\;{\rm Tan}_x(X,\sfd,\mm)=
\big\{(\R^k,\sfd_{\rm Eucl},c_k\mathcal L^k,0^k)\big\}\Big\}
\quad\text{ for every }k\in\N\text{ with }k\leq N.
\]
With a terminology borrowed from \cite{Cheeger-Colding97I} and inspired by geometric measure theory, points in \(\mathcal R_k\) are called \emph{\(k\)-regular} points
of \(X\). Moreover, given any point \(x\in X\) and any \(k\in\N\), we say that an element
\((Y,\varrho,\mathfrak n,y)\in{\rm Tan}_x(X,\sfd,\mm)\) \emph{splits off a factor \(\R^k\)}
provided
\[
(Y,\varrho,\mathfrak n,y)\cong\big(\R^k\times Z,\sfd_{\rm Eucl}\times\sfd_Z,
\mathcal L^k\otimes\mm_Z,(0^k,z)\big),
\]
for some pointed \(\RCD(0,N-k)\) space \((Z,\sfd_Z,\mm_Z,z)\).

In \cite{Mondino-Naber14} uniqueness of tangents (almost everywhere w.r.t. the reference measure $\mm$) was proved together with rectifiability relying on a new $\delta$-splitting via excess theorem (cf. \cite[Theorem 6.7]{Mondino-Naber14} and \cite[Theorem 5.1]{Mondino-Naber14}). Below we provide a new proof of uniqueness of tangents based on the same principle about propagation of regularity but more similar to the one given in \cite{Cheeger-Colding97I} for Ricci limits.

\begin{theorem}[Uniqueness of tangents]\label{thm:uniq_tg}
Let \((X,\sfd,\mm)\) be an \(\RCD(K,N)\) space. Then it holds
\[
\mm\bigg(X\setminus\bigcup_{k\leq N}\mathcal R_k\bigg)=0.
\]
\end{theorem}
\begin{proof}
\textbf{Step 1.}
Fix any \(k\in\N\) with \(k\leq N\). We define the auxiliary sets
\(A_k,A_k'\subseteq X\) as follows:
\begin{itemize}
\item[\(\rm i)\)] \(A_k\) is the set of all points \(x\in X\) such that
\((\R^k,\sfd_{\rm Eucl},c_k\mathcal L^k,0^k)\in{\rm Tan}_x(X,\sfd,\mm)\),
but no other element of \({\rm Tan}_x(X,\sfd,\mm)\) splits off a factor \(\R^k\).
\item[\(\rm ii)\)] \(A'_k\) is the set of all points \(x\in X\) which satisfy
\((\R^k,\sfd_{\rm Eucl},c_k\mathcal L^k,0^k)\in{\rm Tan}_x(X,\sfd,\mm)\)
and \((\R^\ell,\sfd_{\rm Eucl},c_\ell\mathcal L^\ell,0^\ell)\notin{\rm Tan}_x(X,\sfd,\mm)\)
for every \(\ell\in\N\) with \(\ell>k\).
\end{itemize}
Observe that \(\mathcal R_k\subseteq A_k\subseteq A'_k\). The \(\mm\)-measurability
of the sets \(\mathcal R_k,A_k,A'_k\) can be proven adapting the proof
of \cite[Lemma 6.1]{Mondino-Naber14}. It also follows from Theorem \ref{thm:Eucl_tgs}
that \(\mm\big(X\setminus\bigcup_{k\leq N}A'_k\big)=0\).\\
\textbf{Step 2.} We aim to prove that \(\mm(A'_k\setminus A_k)=0\).
We argue by contradiction: suppose \(\mm(A'_k\setminus A_k)>0\).
Then we can find \(x\in A'_k\setminus A_k\) where the iterated tangent property
of Theorem \ref{thm:iterated_tg} holds. Since \(x\notin A_k\), there exists a pointed
\(\RCD(0,N-k)\) space \((Y,\varrho,\mathfrak n,y)\) with \({\rm diam}(Y)>0\) such that
\[
\big(\R^k\times Y,\sfd_{\rm Eucl}\times\varrho,
\mathcal L^k\otimes\mathfrak n,(0^k,y)\big)\in{\rm Tan}_x(X,\sfd,\mm).
\]
Theorem \ref{thm:Eucl_tgs} yields the existence of a point \(z\in Y\)
such that \((\R^\ell,\sfd_{\rm Eucl},c_\ell\mathcal L^\ell,0^\ell)
\in{\rm Tan}_z(Y,\varrho,\mathfrak n)\), for some \(\ell\in\N\) with
\(0<\ell\leq N-k\). This implies that
\[
(\R^{k+\ell},\sfd_{\rm Eucl},c_{k+\ell}\mathcal L^{k+\ell},0^{k+\ell})
\in{\rm Tan}_{(0^k,z)}(\R^k\times Y,\sfd_{\rm Eucl}\times\varrho,
\mathcal L^k\otimes\mathfrak n).
\]
Therefore, Theorem \ref{thm:iterated_tg} guarantees that
\((\R^{k+\ell},\sfd_{\rm Eucl},c_{k+\ell}\mathcal L^{k+\ell},0^{k+\ell})\)
belongs to \({\rm Tan}_x(X,\sfd,\mm)\), which contradicts the fact that \(x\in A'_k\).
Consequently, we have proven that \(\mm(A'_k\setminus A_k)=0\), as desired.\\
\textbf{Step 3.} In order to complete the proof of the statement,
it suffices to show that
\begin{equation}\label{eq:uniq_tg_cl}
\mm\big(B_R(p)\cap(A_k\setminus\mathcal R_k)\big)=0
\quad\text{ for every }p\in X\text{ and }R>0.
\end{equation}
Let \(p\in X\) and \(R,\eta>0\) be fixed.
Choose any \(\delta\in(0,\eta)\) associated with \(\eta\) as in Proposition
\ref{prop:delta-split_to_GH}. Moreover, choose any \(\eps\in(0,1/7)\) associated
with \(\delta^2\) as in Proposition \ref{prop:GH_to_delta-split}.
Given a point \(x\in A_k\), we can find \(r_x\in(0,1)\) such that
\(4r_x^2|K|\leq\eps\) and
\[
\sfd_{\rm pmGH}\Big(\big(X,\sfd/(2r_x),\mm_{2r_x}^x,x\big),
(\R^k,\sfd_{\rm Eucl},c_k\mathcal L^k,0^k)\Big)\leq\eps.
\]
By applying Vitali covering lemma to the family
\(\big\{B_{r_x}(x)\,:\,x\in A_k\cap B_R(p)\big\}\), we obtain a sequence
\((x_i)_i\subseteq A_k\cap B_R(p)\) such that \(\big\{B_{r_{x_i}}(x_i)\big\}_i\)
are pairwise disjoint and \(A_k\cap B_R(p)\subseteq\bigcup_i B_{5r_{x_i}}(x_i)\).
For any \(i\in\N\), we know from
Proposition \ref{prop:GH_to_delta-split} that there exists a \(\delta^2\)-splitting
map \(u^i\colon B_{10r_{x_i}}(x_i)\to\R^k\). Furthermore,
by Proposition \ref{prop:propag_delta-split} there exists a Borel
set \(G_\eta^i\subseteq B_{5r_{x_i}}(x_i)\) such that
\(\mm\big(B_{5r_{x_i}}(x_i)\setminus G_\eta^i\big)\leq
C_N\delta\,\mm\big(B_{5r_{x_i}}(x_i)\big)\) and \(u^i\colon B_s(x)\to\R^k\)
is a \(\delta\)-splitting map for every \(x\in G_\eta^i\) and \(s\in(0,5r_{x_i})\).
Hence, by Proposition \ref{prop:delta-split_to_GH}, for any
\(x\in G_\eta^i\) the following property holds:
\begin{equation}\label{eq:uniq_tg_aux}\begin{split}
&\text{Given any element }(Y,\varrho,\mathfrak n,y)\in{\rm Tan}_x(X,\sfd,\mm),
\text{ there exists}\\
&\text{a pointed }\RCD(0,N-k)\text{ space }(Z,\sfd_Z,\mm_Z,z)\text{ such that }\\
&\sfd_{\rm pmGH}\Big((Y,\varrho,\mathfrak n,y),\big(\R^k\times Z,\sfd_{\rm Eucl}
\times\sfd_Z,\mathcal L^k\otimes\mm_Z,(0^k,z)\big)\Big)\leq\eta.
\end{split}\end{equation}
Then let us define \(G_\eta\coloneqq\bigcup_i G_\eta^i\). Clearly, each element
of \(G_\eta\) satisfies \eqref{eq:uniq_tg_aux}. Moreover, it holds
\begin{equation}\label{eq:uniq_tg_aux2}\begin{split}
\mm\big(B_R(p)\cap(A_k\setminus G_\eta)\big)&\leq
\sum_{i\in\N}\mm\big(B_{5r_{x_i}}(x_i)\setminus G_\eta^i\big)
\leq C_N\delta\sum_{i\in\N}\mm\big(B_{5r_{x_i}}(x_i)\big)\\
&\leq C_N\eta\sum_{i\in\N}\mm\big(B_{r_{x_i}}(x_i)\big)
\leq C_N\eta\,\mm\big(B_{R+1}(p)\big).
\end{split}\end{equation}
Now consider the Borel set \(G\coloneqq\bigcap_i\bigcup_j G_{1/2^{i+j}}\). It follows
from \eqref{eq:uniq_tg_aux2} that \(\mm\big(B_R(p)\cap(A_k\setminus G)\big)=0\).
Moreover, let \(x\in A_k\cap G\) and
\((Y,\varrho,\mathfrak n,y)\in{\rm Tan}_x(X,\sfd,\mm)\)
be fixed. Then by using \eqref{eq:uniq_tg_aux} we can find a sequence
\(\big\{(Z_i,\sfd_{Z_i},\mm_{Z_i},z_i)\big\}_i\) of pointed \(\RCD(0,N-k)\)
spaces such that
\begin{equation}\label{eq:uniq_tg_aux3}
\big(\R^k\times Z_i,\sfd_{\rm Eucl}\times\sfd_{Z_i},\mathcal L^k\otimes
\mm_{Z_i},(0^k,z_i)\big)\overset{\rm pmGH}\longrightarrow(Y,\varrho,\mathfrak n,y)
\quad\text{ as }i\to\infty.
\end{equation}
Up to a not relabelled subsequence, we can suppose that
\((Z_i,\sfd_{Z_i},\mm_{Z_i},z_i)\to(Z,\sfd_Z,\mm_Z,z)\) in the pmGH-topology,
for some pointed \(\RCD(0,N-k)\) space \((Z,\sfd_Z,\mm_Z,z)\). Consequently,
\eqref{eq:uniq_tg_aux3} ensures that \((Y,\varrho,\mathfrak n,y)\) is isomorphic to
\(\big(\R^k\times Z,\sfd_{\rm Eucl}\times\sfd_Z,\mathcal L^k\otimes\mm_Z,(0^k,z)\big)\).
Given that \(x\in A_k\), we deduce that \(Z\) must be a singleton. In other words,
we have proven that any element of \({\rm Tan}_x(X,\sfd,\mm)\) is isomorphic to
\((\R^k,\sfd_{\rm Eucl},c_k\mathcal L^k,0^k)\), so that \(x\in\mathcal R_k\).
This shows that \(A_k\cap G\subseteq\mathcal R_k\), whence
the claim \eqref{eq:uniq_tg_cl} follows. 
\end{proof}
By combining Theorem \ref{thm:uniq_tg} with the properties of \(\delta\)-splitting
maps discussed in Section \ref{s:splitting_maps}, we can give a direct proof of the following
result, that was proved for the first time in \cite{Kitabeppu2019}:
\begin{theorem}\label{thm:lscdimversion}
Let \((X,\sfd,\mm)\) be an \(\RCD(K,N)\) space. Let \(k\in\N\), \(k\leq N\)
be the maximal number such that \(\mm(\mathcal R_k)>0\). Then for any \(x\in X\)
and \(\ell>k\) we have that no element of \({\rm Tan}_x(X,\sfd,\mm)\) splits off
a factor \(\R^\ell\). In particular, it holds that
\(\mathcal R_\ell=\emptyset\) for every \(\ell>k\).
\end{theorem}
\begin{proof}
First of all, we claim that for any given \(\ell>k\) there exists \(\eps>0\) such that
\begin{equation}\label{eq:no_regular_pt_higher_aux1}
\sfd_{\rm pmGH}\Big((\R^j,\sfd_{\rm Eucl},c_j\mathcal L^j,0^j),
\big(\R^\ell\times Z,\sfd_{\rm Eucl}\times\sfd_Z,\mathcal L^\ell\otimes\mm_Z,(0^\ell,z)\big)\Big)>\eps
\end{equation}
for every \(j\leq k\) and for every pointed \(\RCD(0,N-\ell)\) space
\((Z,\sfd_Z,\mm_Z,z)\). This can be easily checked arguing by contradiction.

We prove the main statement by contradiction: suppose there exist
\(x\in X\) and \(\ell>k\) such that 
\begin{equation}\label{eq:no_regular_pt_higher_aux2}
\big(\R^\ell\times Z,\sfd_{\rm Eucl}\times\sfd_Z,\mathcal L^\ell\otimes\mm_Z,(0^\ell,z)\big)\in{\rm Tan}_x(X,\sfd,\mm)
\end{equation}
for some pointed \(\RCD(0,N-\ell)\) space \((Z,\sfd_Z,\mm_Z,z)\). Consider
\(\eps>0\) associated with \(\ell\) as in \eqref{eq:no_regular_pt_higher_aux1}.
Choose \(\delta>0\) associated with \(\eps\) as in Proposition
\ref{prop:delta-split_to_GH}, then \(\eta>0\) associated with
\(\delta^2\) as in Proposition \ref{prop:GH_to_delta-split}. It follows
from \eqref{eq:no_regular_pt_higher_aux2} that there is \(r>0\)
such that \(r^2|K|\leq\eta\) and
\[
\sfd_{\rm pmGH}\Big((X,\sfd/r,\mm_r^x,x),\big(\R^\ell\times Z,
\sfd_{\rm Eucl}\times\sfd_Z,\mathcal L^\ell\otimes\mm_Z,(0^\ell,z)\big)\Big)\leq\eta.
\]
Then Proposition \ref{prop:GH_to_delta-split} guarantees the existence of a
\(\delta^2\)-splitting map \(u\colon B_{5r}(x)\to\R^\ell\). Therefore, by
Propositions \ref{prop:propag_delta-split} and \ref{prop:delta-split_to_GH}
we know that there exists a Borel set \(G\subseteq B_r(x)\) with \(\mm(G)>0\)
satisfying the following property: for any point \(y\in G\), it holds that each
element of \({\rm Tan}_y(X,\sfd,\mm)\) is \(\eps\)-close (with respect
to the distance \(\sfd_{\rm pmGH}\)) to some space that splits off a
factor \(\R^\ell\). Given that \(X\setminus(\mathcal R_1\cup\cdots\cup\mathcal R_k)\)
has null \(\mm\)-measure by Theorem \ref{thm:uniq_tg}, there must exist \(y\in G\)
and \(j\leq k\) for which \((\R^j,\sfd_{\rm Eucl},c_j\mathcal L^j,0^j)\) is
the only element of \({\rm Tan}_y(X,\sfd,\mm)\). Consequently, we have that
\[
\sfd_{\rm pmGH}\Big((\R^j,\sfd_{\rm Eucl},c_j\mathcal L^j,0^j),
\big(\R^\ell\times Z',\sfd_{\rm Eucl}\times\sfd_{Z'},\mathcal L^\ell\otimes\mm_{Z'},
(0^\ell,z')\big)\Big)\leq\eps
\]
for some pointed \(\RCD(0,N-\ell)\) space \((Z',\sfd_{Z'},\mm_{Z'},z')\).
This is in contradiction with \eqref{eq:no_regular_pt_higher_aux1}.
\end{proof}
\begin{remark}[Constant dimension]{\rm
We point out that the first and third named authors proved in \cite{BS18}
that any \(\RCD(K,N)\) space \((X,\sfd,\mm)\) has `constant dimension',
in the following sense: there exist a (unique) \(k\in\N\), \(k\leq N\)
such that \(\mm(X\setminus\mathcal R_k)=0\). The number \(k\) is called \emph{essential dimension} of \((X,\sfd,\mm)\) and denoted by $\dim(X,\sfd,\mm)$. With this notation, Theorem \ref{thm:lscdimversion} can be rephrased by saying that at no point of $X$ an element of the tangent cone can split off a Euclidean factor of dimension bigger than $\dim(X,\sfd,\mm)$.
}\end{remark}

Actually, Theorem \ref{thm:lscdimversion} above is an instance of a more general result that can be proved arguing in a similar manner: the essential dimension of $\RCD(K,N)$ spaces is lower semicontinuous with respect to pointed measured Gromov-Hausdorff convergence.\\
This statement has been proved for the first time in \cite[Theorem 4.10]{Kitabeppu2019}. Below we just sketch how our techniques can provide a slightly more direct proof, still based on the same ideas and on the theory of convergence of Sobolev functions on varying spaces developed in \cite{AmbrosioHonda17,AmbrosioHonda18}.\\
Let us point out that the result below is independent of \cite{BS18} once the essential dimension of an $\RCD(K,N)$ m.m.s. is understood as the maximal $n$ for which $\mm(\mathcal{R}_n)>0$.

\begin{theorem}\label{thm:lscdimension}
Let $(X_n,\sfd_n,\mm_n,x_n)$ and $(X,\sfd,\mm,x)$ be pointed $\RCD(K,N)$ metric measure spaces and assume that $(X_n,\sfd_n,\mm_n,x_n)$ converge to $(X,\sfd,\mm,x)$ in the pointed measured Gromov-Hausdorff sense. Then
\begin{equation*}
\dim(X,\sfd,\mm)\le\liminf_{n\to\infty}\dim(X_n,\sfd_n,\mm_n).
\end{equation*}
\end{theorem}   

\begin{proof}
Let $k:=\dim(X,\sfd,\mm)$. We need to prove that, for $n$ sufficiently large, it holds $k\le\dim(X_n,\sfd_n,\mm_n)$.

Up to scaling of the distance $\sfd$ on $X$, we can assume that $K\ge -1$ and by Proposition \ref{prop:GH_to_delta-split} we find $y\in X$ and a $\delta$-splitting map $u:B_2(y)\to \R^k$. Arguing as in the proof of \cite[Proposition 3.9]{BPS19}, relying on the convergence and stability results of \cite{AmbrosioHonda18}, we can find $1<r<2$, points $X_n\ni y_n\to y\in X$ and $2\delta$-splitting maps $u_n:B_r(y_n)\to\R^k$, for any $n$ sufficiently large (it suffices to approximate the components of $u$ in the strong $H^{1,2}$-sense with harmonic functions).\\
Next, Proposition \ref{prop:propag_delta-split} provides sets $G_n\subset B_{r/2}(y_n)$ such that $\mm_n(B_r(y_n)\setminus G_n)\le C_N\sqrt{2\delta}\mm_n(B_{r/2}(y_n))$ and 
	\[
	u_n\colon B_s(x)\to\R^k\text{ is a $\sqrt{2\delta}$-splitting map, for every $x\in G_n$ and $s\in(0,r/2)$},
	\]
for any $n$ sufficiently large.

Now it suffices to choose $\delta$ such that $\sqrt{2\delta}\le \delta_{\epsilon}$ given by Proposition \ref{prop:delta-split_to_GH} to get that, at any point in $G_n$, any tangent is $\eps$-close to a space splitting a factor $\R^k$. Choosing $\eps$ small enough and arguing as in the proof of Theorem \ref{thm:lscdimversion} above we obtain that $\dim(X_n,\sfd_n,\mm_n)\ge k$ for sufficiently large $n$.
\end{proof}

\subsection{Metric rectifiability of \texorpdfstring{\(\rm RCD\)}{RCD} spaces}\label{subsec:rect}
Aim of this section is to exploit \(\delta\)-splitting maps to show that
\(\RCD(K,N)\) spaces are metrically rectifiable, in the following sense:
\begin{definition}
Given a metric measure space \((X,\sfd,\mm)\), \(k\in\N\) and \(\eps>0\),
we say that a Borel set \(E\subseteq X\) is \emph{\((\mm,k,\eps)\)-rectifiable}
provided there exists a sequence \(\big\{(G_n,u_n)\big\}_n\), where
\(G_n\subseteq X\) are Borel sets satisfying \(\mm\big(X\setminus\bigcup_n G_n\big)=0\)
and the maps \(u_n\colon G_n\to\R^k\) are \((1+\eps)\)-biLipschitz with their images.
\end{definition}
Rectifiability of $\RCD(K,N)$ spaces in the above sense was first proved in \cite[Theorem 1.1]{Mondino-Naber14}. Below we provide a different proof, more in the spirit of the Cheeger-Colding theory for Ricci limits (cf. \cite{Cheeger-Colding97III}) and relying on the connection between $\delta$-splitting maps and $\eps$-isometries.
\begin{lemma}\label{lem:GH-isom_all_scales}
Let \(N>1\) be given. Then for any \(\eta>0\) there exists
\(\delta=\delta_{N,\eta}>0\) such that the following property
holds. If \((X,\sfd,\mm)\) is an \(\RCD(K,N)\) space and
\(u\colon B_r(x)\to\R^k\) is a \(\delta\)-splitting map for
some radius \(r>0\) with \(r^2|K|\leq 1\) and some point \(x\in X\) satisfying
\[
\sfd_{\rm pmGH}\big((X,\sfd/r,\mm_r^x,x),
(\R^k,\sfd_{\rm Eucl},c_k\mathcal L^k,0^k)\Big)<\delta^2,
\]
then it holds that \(u\colon B_r(x)\to\R^k\) is an
\emph{\(\eta r\)-GH isometry}, meaning that
\[
\Big|\big|u(y)-u(z)\big|-\sfd(y,z)\Big|\leq\eta r\quad
\text{ for every }y,z\in B_r(x).
\]
\end{lemma}
\begin{proof}
Thanks to a scaling argument, it suffices to prove the statement
for \(r=1\) and \(|K|\leq 1\). We argue by contradiction:
suppose there exist \(\eta>0\), a sequence of spaces \((X_n,\sfd_n,\mm_n,x_n)\)
and a sequence of maps \(u^n\colon B_1(x_n)\to\R^k\), such that the following
properties are satisfied.
\begin{itemize}
\item[\(\rm i)\)] \((X_n,\sfd_n,\mm_n)\) is an \(\RCD(K,N)\) space.
\item[\(\rm ii)\)] \(u^n\) is a \(1/n\)-splitting map with \(u^n(x_n)=0^k\).
\item[\(\rm iii)\)] It holds that
\(\sfd_{\rm pmGH}\big((X_n,\sfd_n,\mm_n,x_n),
(\R^k,\sfd_{\rm Eucl},c_k\mathcal L^k,0^k)\big)\leq 1/n\).
\item[\(\rm iv)\)] \(u^n\) is not an \(\eta\)-GH isometry, so that
there exist points \(y_n,z_n\in B_1(x_n)\) such that
\begin{equation}\label{eq:u_n_no_GH-isom}
\Big|\big|u^n(y_n)-u^n(z_n)\big|-\sfd_n(y_n,z_n)\Big|>\eta.
\end{equation}
\end{itemize}
Observe that item iii) guarantees that
\((X_n,\sfd_n,\mm_n,x_n)\to(\R^k,\sfd_{\rm Eucl},c_k\mathcal L^k,0^k)\)
in the pmGH-topology. Possibly taking a not relabelled subsequence,
it holds that \(u^n\to u^\infty\) strongly in \(H^{1,2}\) on \(B_1(0^k)\),
for some limit map \(u^\infty\colon B_1(0^k)\to\R^k\)
(cf.\ \cite{AmbrosioHonda17,AmbrosioHonda18}
for the theory of convergence on varying spaces).\\
We also deduce from item ii) above that \({\rm Hess}(u^\infty_a)=0\)
and \(\nabla u^\infty_a\cdot\nabla u^\infty_b=\delta_{ab}\) on \(B_1(0^k)\)
for all \(a,b=1,\ldots,k\) (further details are discussed in the proof of \cite[Proposition 3.7]{BPS19}), whence \(u^\infty\) is the restriction
to \(B_1(0^k)\) of an orthogonal transformation of \(\R^k\). This gives a contradiction since, by letting \(n\to\infty\) in
\eqref{eq:u_n_no_GH-isom}, we obtain that
\[
\Big|\big|u^\infty(y_\infty)-u^\infty(z_\infty)\big|-
|y_\infty-z_\infty|\Big|\geq\eta,
\]
where \(y_\infty,z_\infty\in B_1(0^k)\) stand for the limit points of \((y_n)_n\)
and \((z_n)_n\), respectively (notice that $x_{\infty}\neq y_{\infty}$ as a consequence of \eqref{eq:u_n_no_GH-isom} and (i) in Definition \ref{def:splitting maps}).
\end{proof}
Let \((X,\sfd,\mm)\) be an \(\RCD(K,N)\) space. Let \(k\in\N\) be
such that \(k\leq N\). Then we define
\[
(\mathcal R_k)_{r,\delta}\coloneqq\Big\{x\in\mathcal R_k\;\Big|\;
\sfd_{\rm pmGH}\big((X,\sfd/s,\mm_s^x,x),
(\R^k,\sfd_{\rm Eucl},c_k\mathcal L^k,0^k)\big)<\delta\;\text{ for every }s<r\Big\}
\]
for every \(r,\delta>0\). Observe that for any given \(\delta>0\) it holds
that \((\mathcal R_k)_{r,\delta}\nearrow\mathcal R_k\) as \(r\searrow 0\).
\begin{theorem}[Rectifiability of \(\RCD\) spaces]\label{thm:rect}
Let \((X,\sfd,\mm)\) be an \(\RCD(K,N)\) space. Let \(k\in\N\) be
such that \(k\leq N\). Then the \(k\)-regular set \(\mathcal R_k\)
of \(X\) is \((\mm,k,\eps)\)-rectifiable for every \(\eps>0\).
\end{theorem}
\begin{proof}
We claim that for any $\eps>0$ there exists an $(\mm,k,\eps)$-rectifiable set $G^{\eps}\subset \mathcal{R}_k$ such that $\mm(\mathcal{R}_k\setminus G^{\eps})<\eps$.\\
Notice that the statement follows from the claim above observing that
\begin{equation*}
\mm\Big(\mathcal{R}_k\setminus \bigcup_{n=1}^{\infty}G^{\eps/n}\Big)=0.	
\end{equation*}

Let us prove the claim in two steps.\\
\textbf{Step 1.}
We claim that for any \(\eta>0\) there exists
\(\delta=\delta_{N,\eta}\in(0,1)\) such that the following property holds:
if \((X,\sfd,\mm)\) is an \(\RCD(K,N)\) space and \(u\colon B_{5r}(p)\to\R^k\)
is a \(\delta\)-splitting map for some radius \(r>0\) satisfying \(r^2|K|\leq 1\)
and some point \(p\in(\mathcal R_k)_{2r,\delta}\), then there exists a Borel set
\(G\subseteq B_r(p)\) such that
\(\mm\big(B_r(p)\setminus G\big)\leq C_N\eta\,\mm\big(B_r(p)\big)\) and
\begin{equation}\label{eq:rect_aux}
\Big|\big|u(x)-u(y)\big|-\sfd(x,y)\Big|\leq\eta\,\sfd(x,y)
\quad\text{ for every }x,y\in(\mathcal R_k)_{2r,\delta}\cap G.
\end{equation}

To prove it, choose any \(\delta\in(0,\eta^2)\) so that \(\sqrt\delta\) is associated
with \(\eta\) as in Lemma \ref{lem:GH-isom_all_scales}. Now let us consider
an \(\RCD(K,N)\) space \((X,\sfd,\mm)\) and a \(\delta\)-splitting
map \(u\colon B_{5r}(p)\to\R^k\), for some \(r>0\) with \(r^2|K|\leq 1\)
and \(p\in(\mathcal R_k)_{2r,\delta}\). By Proposition
\ref{prop:propag_delta-split}, we can find a Borel set \(G\subseteq B_r(p)\) such
that \(\mm\big(B_r(p)\setminus G\big)\leq C_N\eta\,\mm\big(B_r(p)\big)\) and
\(u\colon B_s(x)\to\R^k\) is a \(\sqrt\delta\)-splitting map for all \(x\in G\)
and \(s\in(0,2r)\). Then Lemma \ref{lem:GH-isom_all_scales} guarantees that
the map \(u\colon B_s(x)\to\R^k\) is an \(\eta s\)-GH isometry for every
\(x\in(\mathcal R_k)_{2r,\delta}\cap G\) and \(s\in(0,2r)\) (here we used
the fact that \(x\in(\mathcal R_k)_{2r,\delta}\subseteq(\mathcal R_k)_{s,\delta}\)).

Fix any \(x,y\in(\mathcal R_k)_{2r,\delta}\cap G\).
Being \(\sfd(x,y)<2r\), we know that the map
\(u\colon B_{\sfd(x,y)}(x)\to\R^k\) is an \(\eta\,\sfd(x,y)\)-GH isometry,
thus in particular \(\big||u(x)-u(y)|-\sfd(x,y)\big|\leq\eta\,\sfd(x,y)\).
This yields \eqref{eq:rect_aux}.\\
\textbf{Step 2.} Fix $\bar x\in X$, $R>0$, $\eps>0$. We aim to build an $(\mm,k,\eps)$-rectifiable set $G$ satisfying $\mm(B_R(\bar x)\cap \mathcal{R}_k\setminus G)<\eps$. Note that this easily implies our claim.

Let $\eta<\eps$ to be chosen later, $\delta=\delta_{N,\eta}$ according to Step 1, $\bar \eps\in (0,\delta)$ associated to $\delta$ as in Proposition \ref{prop:GH_to_delta-split} and $r>0$ satisfying $r^2|K|\le 1$ and $\mm(B_R(\bar x)\cap (\mathcal{R}_k\setminus (\mathcal{R})_{2r,\bar \eps}))\le \eps/2$.
By Vitali covering lemma, we find
points \(x_1,\ldots,x_\ell\in B_R(\bar x)\cap(\mathcal R_k)_{2r,\bar \eps}\) for which
\(\big\{B_{r/5}(x_i)\big\}_{i=1}^\ell\) are pairwise disjoint and
\(B_R(\bar x)\cap(\mathcal R_k)_{2r,\bar \eps}\subseteq B_{r}(x_1)\cup\cdots\cup B_{r}(x_\ell)\).
Proposition \ref{prop:GH_to_delta-split} guarantees the existence of a
\(\delta\)-splitting map \(u^i\colon B_{5r}(x_i)\to\R^k\) for every
\(i=1,\ldots,\ell\). Therefore Step 1 yields Borel sets
\(G_i\subseteq B_{r}(x_i)\) such that
\(\mm\big(B_{r}(x_i)\setminus G_i\big)\leq C_N\eta\,\mm\big(B_{r}(x_i)\big)\)
and \(\big||u^i(x)-u^i(y)|-\sfd(x,y)\big|\leq\eta\,\sfd(x,y)\) for
every \(x,y\in (\mathcal R_k)_{2r,\bar \eps}\cap G_i\), for every $i=1,\ldots,\ell$.

Since $\eta<\eps$, we deduce that
\(u^i\) is \((1+\eps)\)-biLipschitz with its image when restricted to
\((\mathcal R_k)_{2r,\bar \eps}\cap G_i\), whence \(G\coloneqq (\mathcal R_k)_{2r,\bar \eps}\cap\bigcup_{i=1}^\ell G_i\)
is \((\mm,k,\eps)\)-rectifiable. Observe that
\[\begin{split}
\mm\Big(\big(B_R(\bar x)\cap(\mathcal R_k)_{2r,\bar \eps}\big)\setminus G\Big)
&\leq\sum_{i=1}^\ell\mm\big(B_{r}(x_i)\setminus G_i\big)
\leq C_N\eta\sum_{i=1}^\ell\mm\big(B_{r}(x_i)\big)\\
&\leq C_N\eta\sum_{i=1}^\ell\mm\big(B_{r/5}(x_i)\big)
\leq C_N\eta\,\mm\big(B_{R+1}(\bar x)\big).
\end{split}\]
Choosing $\eta>0$ such that $C_N\eta\,\mm\big(B_{R+1}(\bar x)\big)<\eps/2$ we get the sought conclusion.
\end{proof}
\subsection{Behaviour of the reference measure under charts}\label{subsec:measure}
Aim of this subsection is to prove absolute continuity of the reference measure $\mm$ of an $\RCD(K,N)$ metric measure space $(X,\sfd,\mm)$ with respect to the relevant Hausdorff measure. This result was first proved in the three independent works \cite{DPMR16,MK16,GP16-2}, heavily relying on \cite{DPR}. The strategy of our proof is essentially taken from \cite{GP16-2}, the main technical simplification being that the charts providing rectifiability in our case are harmonic (indeed they are $\delta$-splitting maps), while in \cite{GP16-2} they were distance functions.

Let us introduce the notation we are going to use in this subsection.\\
Let \(X,Y\) be Polish spaces. Fix a finite Borel measure \(\mu\geq 0\)
on \(X\) and a Borel  map \(\varphi\colon X\to Y\). We shall denote by $\varphi_*$ the pushforward operator, which sends finite Borel measures on $X$ into finite Borel measures on $Y$. Then we define
\begin{equation}\label{eq:def_Pr_functions}
{\sf Pr}_\varphi(f)\coloneqq\frac{\d\varphi_*(f\mu)}{\d\varphi_*\mu}
\quad\text{ for every }f\in L^1(\mu),
\end{equation}
where we adopted the usual notation of geometric measure theory for the density of a measure absolutely continuous with respect to another measure.
The resulting map \({\sf Pr}_\varphi\colon L^1(\mu)\to L^1(\varphi_*\mu)\)
is linear and continuous. Given any \(p\in(1,\infty]\), it holds that
\({\sf Pr}_\varphi\) maps continuously \(L^p(\mu)\) to \(L^p(\varphi_*\mu)\).
The \emph{essential image} of a Borel set \(E\subseteq X\) is defined as
\({\sf Im}_\varphi(E)\coloneqq\big\{{\sf Pr}_\varphi(\nchi_E)>0\big\}\subseteq Y\).
\begin{proposition}[Differential of an \(\R^k\)-valued Lipschitz map]\label{prop:D_phi}
Let \((X,\sfd,\mu)\) be an infinitesimally Hilbertian metric measure space
such that \(\mu\) is finite. Let \(\varphi\colon X\to\R^k\) be a Lipschitz map.
Then there exists a unique linear and continuous operator
\({\sf D}_\varphi\colon L^2_\mu(TX)\to L^2(\R^k,\R^k;\varphi_*\mu)\) such that
\begin{equation}\label{eq:def_D_phi}
\int_F\nabla f\cdot{\sf D}_\varphi(v)\,\d\varphi_*\mu
=\int_{\varphi^{-1}(F)}\nabla_\mu(f\circ\varphi)\cdot v\,\d\mu
\quad\forall f\in C^\infty_c(\R^k),\,v\in L^2_\mu(TX),\,F\subseteq\R^k
\text{ Borel.}
\end{equation}
In particular, if \(v\in D({\rm div}_\mu)\), then the distributional divergence
of \({\sf D}_\varphi(v)\) is given by \({\sf Pr}_\varphi\big({\rm div}_\mu(v)\big)\).

Moreover, if the map \(\varphi\) is biLipschitz with its image when restricted to some
Borel set \(E\subseteq X\) and \(v_1,\ldots,v_k\in L^2_\mu(TX)\) are independent on \(E\),
then the vectors
\({\sf D}_\varphi(\nchi_E\,v_1)(y),\ldots,{\sf D}_\varphi(\nchi_E\,v_k)(y)\)
constitute a basis of \(\R^k\) for \(\varphi_*\mu\)-a.e.\ point
\(y\in{\sf Im}_\varphi(E)\).
\end{proposition}
\begin{proof}
Existence of the map \({\sf D}_\varphi\) is proven in \cite{GP16-2}:
with the terminology used therein, it suffices to define
\({\sf D}_\varphi\coloneqq\iota\circ{\sf Pr}_\varphi\circ\d\varphi\).
The fact that this map satisfies \eqref{eq:def_D_phi} follows from
\cite[Proposition 2.7]{GP16-2} and the very definition of \(\iota\)
(we do not need to require properness of \(\varphi\), as \(\mu\) is a finite measure).
Uniqueness of \({\sf D}_\varphi\) follows from the fact that
\(\big\{\nabla f\,:\,f\in C^\infty_c(\R^k)\big\}\)
generates \(L^2(\R^k,\R^k;\varphi_*\mu)\). Now suppose \(v\in D({\rm div}_\mu)\).
Then for every \(f\in C^\infty_c(\R^k)\) it holds that
\(f\circ\varphi\in H^{1,2}(X,\sfd,\mu)\), whence
\[\begin{split}
\int\nabla f\cdot{\sf D}_\varphi(v)\,\d\varphi_*\mu
&\overset{\eqref{eq:def_D_phi}}=\int\nabla_\mu(f\circ\varphi)\cdot v\,\d\mu
=-\int f\circ\varphi\,{\rm div}_\mu(v)\,\d\mu
=-\int f\,\d\varphi_*\big({\rm div}_\mu(v)\mu\big)\\
&\overset{\eqref{eq:def_Pr_functions}}=
-\int f\,{\sf Pr}_\varphi\big({\rm div}_\mu(v)\big)\,\d\varphi_*\mu.
\end{split}\]
This shows that the distributional divergence of \({\sf D}_\varphi(v)\)
is represented by \({\sf Pr}_\varphi\big({\rm div}_\mu(v)\big)\).
Finally, the last claim of the statement follows from \cite[Proposition 2.2]{GP16-2}
and \cite[Proposition 2.10]{GP16-2}.
\end{proof}
\begin{theorem}[Behaviour of \(\mm\) under charts]
Let \((X,\sfd,\mm)\) be an \(\RCD(K,N)\) space.
Consider a \(\delta\)-splitting map \(u\colon B_r(p)\to\R^k\)
which is \((1+\eps)\)-biLipschitz with its image (for some
\(\eps<1/k\)) when restricted to some compact set
\(K\subseteq B_r(p)\). Then it holds that
\[
u_*(\mm|_K)\ll\mathcal L^k.\]
In particular, for any \(k\in\N\), \(k\leq N\), \(\mm|_{\mathcal R_k}\)
is absolutely continuous with respect to the \(k\)-dimensional Hausdorff measure
on \((X,\sfd)\).
\end{theorem}
\begin{proof}
First of all, fix a good cut-off function \(\eta\colon X\to\R\) for the pair
\(K\subseteq B_r(p)\), in the sense of Lemma \ref{lem:good_cut-off}.
Define \(\mu\coloneqq\mm|_{B_r(p)}\) and \(\varphi\coloneqq\eta u\colon X\to\R^k\).
Observe that the components \(\varphi_1,\ldots,\varphi_k\)  of \(\varphi\) are test
functions and \(\varphi|_K\) is \((1+\eps)\)-biLipschitz with its image. Consider the
differential \({\sf D}_\varphi\colon L^2_\mu(TX)\to L^2(\R^k,\R^k;\varphi_*\mu)\)
defined in Proposition \ref{prop:D_phi}. Fix a sequence \((\psi_i)_i\)
of compactly-supported, Lipschitz functions \(\psi_i\colon X\to[0,1]\)
that pointwise converge to \(\nchi_K\). We then set
\[
v^i_a\coloneqq{\sf D}_\varphi(\psi_i\nabla_\mu\varphi_a)\in
L^2(\R^k,\R^k;\varphi_*\mu)\quad\text{ for every }i\in\N\text{ and }a=1,\ldots,k.
\]
Note that \(\psi_i\nabla_\mu \varphi_a\in D({\rm div}_\mu)\)
by the Leibniz rule for divergence and the fact that
\(\varphi_a\in D(\Delta_\mu)\), whence Proposition \ref{prop:D_phi}
ensures that the distributional divergence of each vector field
\(v^i_a\) is an \(L^2(\varphi_*\mu)\)-function.
Hence, it holds that \(\mathcal I_{ia}\coloneqq v^i_a\,\varphi_*\mu\)
is a normal \(1\)-current in \(\R^k\) (see \cite[Corollary 2.12]{GP16-2}). Note also that
\[
\overrightarrow{\mathcal I_{ia}}=\nchi_{\{|v^i_a|>0\}}\frac{v^i_a}{|v^i_a|}
\;\text{ and }\;\|\mathcal I_{ia}\|=|v^i_a|\,\varphi_*\mu\quad\text{ for every }
i\in\N\text{ and }a=1,\ldots,k.
\]
Call \(A_i\) the set of \(y\in\R^k\) such that
\(v^i_1(y),\ldots,v^i_k(y)\) form a basis of \(\R^k\).
Since \((\varphi_*\mu)|_{A_i}\ll\|\mathcal I_{ia}\|\) holds for all \(a=1,\ldots,k\),
by applying \cite[Corollary 1.12]{DPR} we deduce that
\begin{equation}\label{eq:meas_ac_aux}
(\varphi_*\mu)|_{A_i}\ll\mathcal L^k\quad\text{ for every }i\in\N.
\end{equation}
Now define \(v_a\coloneqq{\sf D}_\varphi(\nchi_K\nabla_\mu \varphi_a)
\in L^2(\R^k,\R^k;\varphi_*\mu)\) for every \(a=1,\ldots,k\).
It can readily checked that \(\nabla_\mu\varphi_1,\ldots,\nabla_\mu\varphi_k\)
are independent on \(K\) (here the assumption \(\eps<1/k\) plays a role),
whence the vectors \(v_1(y),\ldots,v_k(y)\) are linearly independent for
\(\varphi_*\mu\)-a.e.\ \(y\in{\sf Im}_\varphi(K)\) by Proposition \ref{prop:D_phi}.\\
Furthermore, for any given \(j=1,\ldots,k\),
we can see (by using dominated convergence theorem) that
\(\psi_i\nabla_\mu\varphi_a\to\nchi_K\nabla_\mu\varphi_a\) in
\(L^2_\mu(TX)\) as \(i\to\infty\), thus \(v^i_a\to v_a\)
in \(L^2(\R^k,\R^k;\varphi_*\mu)\) as \(i\to\infty\) by continuity of
\({\sf D}_\varphi\). In particular, possibly passing to a not
relabelled subsequence, we can assume that \(\lim_i v^i_a(y)=v_a(y)\)
for \(\varphi_*\mu\)-a.e.\ \(y\in\R^k\). This implies that
\((\varphi_*\mu)\big({\sf Im}_\varphi(K)\setminus\bigcup_i A_i\big)=0\), thus
\eqref{eq:meas_ac_aux} yields \((\varphi_*\mu)|_{{\sf Im}_\varphi(K)}\ll\mathcal L^k\).
Since \({\sf Im}_\varphi(K)=\big\{{\sf Pr}_\varphi(\nchi_K)>0\big\}\)
by definition, we conclude that
\[
u_*(\mm|_K)=\varphi_*(\mu|_K)=\frac{\d\varphi_*(\nchi_K\mu)}{\d\varphi_*\mu}\,\varphi_*\mu
={\sf Pr}_\varphi(\nchi_K)\,\varphi_*\mu\ll\mathcal L^k.
\]
Therefore, the first part of the statement is finally achieved.

The second part of the statement follows from the first one, the
inner regularity of \(\mm\) and (the proof of) Theorem \ref{thm:rect}.
\end{proof}

\def\cprime{$'$} \def\cprime{$'$}

\end{document}